\documentclass[12pt]{amsart}

\usepackage{amsmath}
\usepackage{amsfonts}
\usepackage{multirow}
\usepackage{mathtools}
\usepackage[english]{babel}
\usepackage{hyperref}
\usepackage[skip=9pt,indent=10pt]{parskip}

\newtheorem{theorem}{Theorem}
\newtheorem{lemma}[theorem]{Lemma}

\newtheorem*{corollary*}{Corollary}

\theoremstyle{definition}
 
\theoremstyle{remark}

\newcommand{\Q}{{\mathbb{Q}}}

\newcommand{\wenvert}[1]{\left\lvert\left\lvert#1\right\rvert\right\rvert}
\newcommand{\norml}[1]{\wenvert{#1}_1}
\newcommand{\floor}[1]{\left\lfloor#1\right\rfloor}
\newcommand{\Mod}[1]{\ \left(\mathrm{mod}\ #1\right)}

\let\abs=\envert

\begin{document}
\title[Gauss-Kra\"{i}tchik formula for cyclotomic polynomials]
{On Gauss-Kra\"{i}tchik formula for cyclotomic polynomials via symmetric functions}
\author[Tomohiro Yamada]{Tomohiro Yamada*}
\address{Center for Japanese language and culture, Osaka University,
562-8678, 3-5-10, Semba-Higashi, Minoo, Osaka, Japan}
\email{tyamada1093@gmail.com}


\begin{abstract}
We give explicit upper bounds for coefficients of polynomials appearing in Gauss-Kra\"{i}tchik formula
for cyclotomic polynomials.
We use a certain relation between elementary symmetric polynomials and power sums polynomials.
\end{abstract}

\maketitle

\section{Introduction}\label{intro}

The $d$-th cyclotomic polynomial
\begin{equation}
\Phi_d(X)=\prod_{1\leq k\leq d, \gcd(k, d)=1} (X-\zeta_d^k),
\end{equation}
where $\zeta_d=e^{2\pi i/d}$,
is the monic polynomial of integer coefficients whose roots are exactly $d$-th primitive roots of unity.
We expand $\Phi_d(X)$ into
\begin{equation}
\Phi_d(X)=\sum_{k=0}^{\varphi(d)} c_{d, k} X^{\varphi(d)-k},
\end{equation}
where $\varphi(d)$ is the Euler totient of $d$.

Many results are known for coefficients $c_{d, k}$ of cyclotomic polynomials.
We begin by noting that the coefficients $c_{d, k}$'s are palindromic in the sense that
$c_{d, k}=c_{d, \varphi(d)-k}$ for any $d\geq 2$ and $0\leq k\leq \varphi(d)$.
Bachman \cite{Bac} proved that
\begin{equation}
\log \max_d \abs{c_{d, k}}=\frac{Ck^\frac{1}{2}}{\log^{1/4} k}\left(1+O\left(\frac{\log\log k}{\log^\frac{1}{2} k}\right)\right)
\end{equation}
for a constant $C>0$ explicitly given by a complicated expression.
Bateman \cite{Bat} proved that
\begin{equation}
\log \max_k \abs{c_{d, k}}<\frac{\tau(d)\log d}{2}<\exp \frac{(\log 2+o(1))\log d}{\log\log d}
\end{equation}
for sufficiently large integers $d$ and Vaughan \cite{Vau} proved that
\begin{equation}
\log \log \max_k \abs{c_{d, k}}>\frac{\log 2 \log d}{\log\log d}
\end{equation}
for infinitely many positive integers $d$,
where $\tau(d)$ denotes the number of divisors of $d$.
Sanna \cite{San} gave a survey for results on coefficients of cyclotomic polynomials.

Gauss' formula in Sections 356--357, pp. 633--638 of \textit{Disquitiones Arithmeticae}
(or pp. 436--440 of the translation \cite{Gau}) states that if $d$ is an odd prime, then we can have
\begin{equation}\label{eq11}
4\Phi_d(X)=\Psi_d(X)^2-D\Xi_d(X)^2,
\end{equation}
where $\Psi_d(X)$ is a polynomial of degree $(d-1)/2$ and of integer coefficients beginning with two,
$\Xi_d(X)$ is a polynomial of degree $(d-3)/2$ and of integer coefficients beginning with one,
and $D=(-1)^{(d-1)/2} d$.
Later, Kra\"{i}tchik proved that \eqref{eq11} holds for any odd squarefree $d$
with $\Psi_d(X)$ of degree $\varphi(d)/2$ and $\Xi_d(X)$ of degree $\varphi(d)/2-1$
in Tome I, pp. 124--129 of \cite{Kra} (see also p.330 of \cite{Rie}).
Tome II, pp. 2--4 of \cite{Kra} gives coefficients of $\Psi_d(X)$ and $\Xi_d(X)$ for $5\leq d\leq 101$
and Table 33 in pp. 445--451 of \cite{Rie} extends this for $5\leq d\leq 149$.
For example, we have
\begin{equation}
\begin{split}
& \Psi_5(X)=2X^2+X+2, ~ \Xi_5(X)=X, \\
& \Psi_7(X)=2X^3+X^2-X-2, ~ \Xi_7(X)=X^2+X, \\
& \Psi_{11}(X)=2X^5+X^4-2X^3+2X^2-X-2, ~ \Xi_{11}(X)=X^4+X, \\
& \Psi_{13}(X)=2X^6+X^5+4X^4-X^3+4X^2+X+2, \\
& \Xi_{13}(X)=X^5+X^3+X, \ldots
\end{split}
\end{equation}

We write $d^\prime=\varphi(d)/2$ and
\begin{equation}\label{eq12}
\Psi_d(X)=\sum_{n=0}^{d^\prime} a_{d, n}X^{d^\prime -n}, ~
\Xi_d(X)=\sum_{n=1}^{d^\prime} b_{d, n}X^{d^\prime -n}.
\end{equation}

Gauss-Kra\"{i}tchik formula gives representations of values of cyclotomic polynomials
as norms in quadratic fields and is useful for studies of arithmetic properties of such values
and the sum $\sigma(N)$ of divisors of $N$.
For example, in \cite{Ymd1} and \cite{Ymd2}, the author used Gauss' formula to obtain upper bounds
for odd perfect numbers of a certain form, where an upper bound for the ratio $\Xi_d(X)/\Psi_d(X)$ plays
an important role.

However, in contrast to the case of cyclotomic polynomials,
few results are known for coefficients of $\Psi_d(X)$ and $\Xi_d(X)$.
It is known that
$a_{d, 0}=2$ and $b_{d, 1}=1$.
Moreover, we have
$a_{d, n}=(-1)^{d^\prime} a_{d, d^\prime -n}$ for $0\leq n\leq d^\prime$
and $b_{d, n}=\pm b_{d, d^\prime-n}$, where the sign is minus if $d\equiv 3\Mod{4}$ is composite
and plus otherwise.

In this paper, we give upper bounds for coefficients of $\Psi_d(X)$ and $\Xi_d(X)$.

\begin{theorem}\label{th11}
Let $d\geq 5$ be an odd squarefree integer, $0\leq n\leq d^\prime$, and
$F=F_{d, n}=\max\left(\{\varphi(f)/2: f\mid d, 1<f\leq n\}\cup\{\abs{1+\sqrt{D}}/2\}\right)$.
Then,
\begin{equation}
\abs{a_{d, n}+b_{d, n}\sqrt{D}}\leq \frac{2F_{d, n}(F_{d, n}+1)\cdots (F_{d, n}+n-1)}{n!}
\end{equation}
and, putting $G_{d, n}=\max\left(\{\varphi(f)/2: f\mid d, 1<f\leq n\}\cup\{\abs{1+\sqrt{d}}/2\}\right)$,
\begin{equation}
\abs{a_{d, n}}+\abs{b_{d, n}}\sqrt{d}\leq \frac{2G_{d, n}(G_{d, n}+1)\cdots (G_{d, n}+n-1)}{n!}.
\end{equation}
\end{theorem}

\begin{corollary*}
For an odd squarefree integer $d\geq 5$ and an integer $n$ with $0\leq n\leq d^\prime$,
we have
\begin{equation}\label{eq100}
\begin{split}
\abs{a_{d, n}+b_{d, n}\sqrt{d}}<\min \left\{\sqrt{\frac{2e^{1/(6(F_{d, n}+n))}}{e\pi n}}
\left(\frac{e(F_{d, n}+n-1)}{F_{d, n}-1}\right)^{F_{d, n}-\frac{1}{2}}, \right. & \\
\left. \sqrt{\frac{2e^{1/(6(F_{d, n}+n))}}{e\pi (F_{d, n}-1)}}
\left(\frac{e(F_{d, n}+n-1)}{n}\right)^{n+\frac{1}{2}}, 2^{F_{d, n}+n}\right\} & .
\end{split}
\end{equation}
\end{corollary*}

Our argument depends on general relations of symmetric functions
and evaluation of power sums of roots of unity.
Let
\begin{equation}\label{eq13}
U_d^\pm(X)=\prod_{1\leq k\leq d, (k/d)=\pm 1} (X-e^{2\pi i k/d}),
\end{equation}
where $(\cdot /d)$ denotes the Jaoobi symbol.
Then, for an odd squarefree integer $d$, we can write
\begin{equation}\label{eq14}
U_d^+(X)=\sum_{n=0}^{d^\prime} u_{d, n} X^{d^\prime -n}, ~
U_d^-(X)=\sum_{n=0}^{d^\prime} \overline{u_{d, n}} X^{d^\prime -n}
\end{equation}
with coefficients $u_{d, n}$ in $\Q(\sqrt{D})$, where $\overline z$ denote the algebraic conjugate of $z$
in $\Q(\sqrt{D})$ and
\begin{equation}
\Psi_d(X)=U_d^+(X)+U_d^-(X), ~ \Xi_d(X)=\frac{U_d^-(X)-U_d^+(X)}{\sqrt{D}}.
\end{equation}
Moreover, we can easily see that $u_{d, n}=a_{d, n}+b_{d, n}\sqrt{D}$.
These coefficients $u_{d, n}$ are the main objects of our study.
For example, if $d$ is an odd prime, then
\begin{equation}\label{eq15}
u_{d, 0}=1, u_{d, 1}=\frac{1-\sqrt{D}}{2},
\end{equation}
and
\begin{equation}\label{eq16}
u_{d, 2}=\frac{D+3}{8}-\frac{1+(2/d)\sqrt{D}}{4}
=\begin{cases}
\frac{(d+3)/4-\sqrt{d}}{2}, & d\equiv 1\Mod{8}, \\
\frac{-d+3}{8}, & d\equiv 3\Mod{8}, \\
\frac{d+3}{8}, & d\equiv 5\Mod{8}, \\
\frac{(-d+3)/4-\sqrt{d}}{2}, & d\equiv 7\Mod{8},
\end{cases}
\end{equation}
which can be confirmed using the Girard-Newton identity and the formula for quadratic Gauss sums,
as we do later.
In general, coefficients $u_{d, n}$ of $U_d^+(X)$ can be represented as elementary symmetric polynomials.
We can write such polynomials using power sums of roots of unity using the well-known Girard-Newton identity.
The identity itself is rather complicated to obtain our upper bounds (but we still need it in order to
determine signs of some quantities).
Instead, we prove another specialized but simple identity.
Combining with evaluation of power sums of roots of unity,
we can bound sizes of coefficients $u_{d, n}$.

One of our motivation is quantification of how representations of values of cyclotomic polynomials as norms in quadratic fields constrain the relative sizes of conjugate embeddings.
As we wrote before, the author used Gauss' formula to obtain upper bounds in \cite{Ymd1} and \cite{Ymd2}, 
for odd perfect numbers of a certain form, where an upper bound for the ratio $\Xi_d(X)/\Psi_d(X)$ plays
an important role.

For example, 
We also have an approximation for the ratio $\abs{\Xi_d(x)/\Psi_d(x)}$, which generalizes Lemma 2.3
and the note in the corrigendum of \cite{Ymd1}.
\begin{theorem}\label{th12}
Put $G_d=G_{d, \floor{\varphi(d)/4}}$.
If $d\geq 5$ is odd and squarefree and $x>2G_d$, then
\begin{equation}
\abs{\frac{\Xi_d(x)}{\Psi_d(x)}-\frac{1}{2x-\mu(d)}}
<\frac{x}{(2x-\mu(d))\sqrt{d}}\left(\left(1-\frac{1}{x}\right)^{-G_d}-1-\frac{G_d}{x}\right).
\end{equation}
\end{theorem}

This theorem gives an error term of order $O(1/(x^2\sqrt{d}))$
to clarify how representations of values of cyclotomic polynomials as norms in quadratic fields constrain the relative sizes of conjugate embeddings.
We note that it would be possible to have an approximate formula with an error term of order $O(1/(x^n\sqrt{d}))$
for a positive integer $n$ if we have a formula of $u_{d, k}$ like \eqref{eq15} and \eqref{eq16}
for each $k\leq n-1$, which would be useful to study other diophantine and arithmetic problems
involving cyclotomic values or integers of special forms such as $(x^m-1)/(x-1)$ and $x^m-1$.
However, in order to obtain such formulae, we must process several cases according to values of
quadratic residue symbols $(k/d)$ over $k=2, \ldots, n-1$ and the main term would be rather complicated.

\section{Power sums of roots of unity}

We write
$g_{d, k}=\sum_{1\leq a\leq d, \gcd(a, d)=1} (a/d)\zeta_d^{ka}$ for the quadratic Gauss sum,
$h_{d, k}=\sum_{1\leq a\leq d, \gcd(a, d)=1} \zeta_d^{ka}$
for the power sum of primitive $d$-th roots of unity,
and $s_{d, k}=\sum_{1\leq a\leq d, (a/d)=1} \zeta_d^{ka}$.
From Chinese Remainder Theorem, we see that the sum $h_{d, k}$ is multiplicative over integers $d$
and the sum $g_{d, k}$ is ``almost'' multiplicative over integers $d$
in the sense that $h_{dm, k}=h_{d, k}h_{m, k}$ and $g_{dm, k}=\pm g_{d, k}g_{m, k}$ whenever $\gcd(m, d)=1$.
Indeed, we have $g_{dm, k}=(d/m)(m/d) g_{d, k}g_{m, k}=(-1)^{(d-1)(m-1)/4} g_{d, k}g_{m, k}$.

\begin{lemma}\label{lm21}
Assume that $d\geq 3$ is odd and squarefree.
If $\gcd(k, d)=1$, then $s_{d, k}=(\mu(d)+(k/d)\sqrt{D})/2$.
If $\gcd(k, d)=f>1$, then, $s_{d, k}=\mu(d/f)\varphi(f)/2$.
\end{lemma}

\begin{proof}
Let $s_d^\pm=\sum_{\gcd(a, d)=1, (a/d)=\pm 1}\zeta_d^a$.
We can easily see that $h_{d, 1}=s_d^+ + s_d^-$ and $g_{d, 1}=s_d^+ - s_d^-$.
Hence, we obtain
\begin{equation}\label{eq21}
s_d^{\pm} = \frac{h_{d, 1} \pm g_{d, 1}}{2}=\frac{\mu(d)\pm\sqrt{D}}{2},
\end{equation}
where we apply the well-known evaluation of quadratic Gauss sums to $g_{d, 1}$
(see for example Theorem 3.3 of \cite{IK})
We note that we can evaluate $g_{d, 1}$ using Theorem 1 of Chapter 6 of \cite{IR}
observing that $g_{d, 1}$ is multiplicative over integers $d$ as noted above.

If $\gcd(k, d)=1$, then $s_{d, k}=\sum_{(a/d)=1}\zeta_d^{ka}=\sum_{(a/d)=(k/d)}\zeta_d^a$
and hence $s_{d, k}=s_d^{(k/d)}=(\mu(d)+(k/d)\sqrt{D})/2$ as desired.

Now we assume that $\gcd(k, d)=f>1$.
We see that $g_{d, k}=\pm g_{f, k}g_{d/f, k}=0$ since $g_{f, k}=0$.
Hence, we obtain
\begin{equation}
s_{d, k}=\frac{g_{d, k}+h_{d, k}}{2}=\frac{h_{d, k}}{2}=\frac{h_{d/f, k}h_{f, k}}{2}=\frac{\mu(d/f)\varphi(f)}{2},
\end{equation}
as desired.
\end{proof}

For example, we see that $u_{d, 1}=-s_{d, 1}=(1-g_{d, 1})/2$ if $d$ is an odd prime, which gives \eqref{eq15}.

\section{Elementary symmetric polynomials and power sums polynomials}

In this section, we consider symmetric functions of an infinite number of variables rather than
a finite number of variables.
For a positive integer $m$, the $m$-th elementary symmetric polynomial of an infinite number of variables
and the $m$-th power sum polynmial of an infinite number of variables are given by
\begin{equation}
S^{(m)}=\sum_{1\leq i_1<i_2<\cdots<i_m} x_{i_1}x_{i_2}\cdots x_{i_m}, ~ S_m=\sum_{i=1}^\infty x_i^m
\end{equation}
respectively.
Moreover, we define $S^{(0)}=1$.

The Girard-Newton identity
\begin{equation}
mS^{(m)}=S^{(m-1)}S_1-S^{(m-2)}S_2+\cdots +(-1)^{m-1}S_m
\end{equation}
is well-known and can be found in $(2.11^\prime)$ of \cite{Mac}
and Section 3.4 of \cite{Esc}.
Mead \cite{Mea} gave a straightforward proof, which is given in Proposition 12.6 of \cite{Mor}.
Min\'{a}\v{c} \cite{Min} gave another proof.
Boklan \cite{Bok} gave another way to derive the Girard-Newton identity only based on
single variable differentiation.
For example, this identity gives that
\begin{equation}
u_{d, 2}=\frac{s_{d, 1}^2-s_{d, 2}}{2},
\end{equation}
which gives \eqref{eq16} combined with Lemma \ref{lm21} if $d$ is an odd prime.

It is also known that
\begin{equation}
S^{(m)}=(-1)^m\sum_{m_1+2m_2+\cdots +\ell m_\ell=m} \prod_{j=1}^\ell \frac{(-1)^{m_j} S_j^{m_j}}{m_j! j^{m_j}},
\end{equation}
which can be found in $(4)$ of \cite{Bok} and is an equivalent form of $(2.14^\prime)$ of \cite{Mac} and $(7.23)$ of \cite{Sta}.

While these identies themselves are rather complicated,
we can prove a surprisingly simple identity for the special polynomial $R_m(X, X, \ldots, X)$,
which we really use in our proof.
Writing
\begin{equation}
S^{(m)}=R_m(S_1, \ldots, S_m)=
\sum_{\substack{e_1\leq e_2\leq \cdots \leq e_k, \\e_1+\cdots +e_k=m}}
\left((-1)^{m-k} w_{(e_1, e_2, \ldots, e_k)} \prod_{i=1}^k S_{e_i}\right)
\end{equation}
with a polynomial $R_m(X_1, \ldots, X_m)$, the Girard-Newton identity gives that 
\begin{equation}
w_{(e_1, e_2, \ldots, e_{k+1})}=\frac{1}{m}\sum_{i=1}^{k+1}w_{(e_1, e_2, \ldots, e_{k+1})\backslash e_i}
\end{equation}
are positive rational numbers,
where $(e_1, e_2, \ldots, e_{k+1})\backslash e_i$ is the $k$-tuple of $e_j$'s with $1\leq j\leq k+1$ but $j\neq i$.
Moreover, putting 
$v_{m, k}=\sum_{e_1+\cdots +e_k=m}w_{(e_1, e_2, \ldots, e_k)}$
for integers $m\geq k\geq 1$, we can easily see that
\begin{equation}
P_m(X)=\sum_{k=0}^m (-1)^{m-k} v_{m, k} X^k=R_m(X, \ldots, X)
\end{equation}
is the polynomial obtained by replacing all $S_i$'s in $S^{(m)}$ by $X$.

Our key tool is the following identity.
\begin{lemma}\label{lm31}
For each $m=1, 2, \ldots$, we have $P_m(X)=\binom{X}{m}=X(X-1)\cdots (X-m+1)/m!$.
\end{lemma}

\begin{proof}
We begin with
\begin{equation}
\log \prod_{i=1}^\infty (1+x_i T)=\sum_{r=1}^\infty \frac{(-1)^{r-1} S_r T^r}{r}.
\end{equation}
We clearly have $\sum_{m=0}^\infty S^{(m)} T^m=\prod_{i=1}^\infty (1+x_i T)$
and therefore
\begin{equation}\label{eq31}
\sum_{m=0}^\infty R_m(S_1, \ldots, S_m) T^m=\exp \sum_{r=1}^\infty \frac{(-1)^{r-1} S_r T^r}{r}.
\end{equation}
As is well known, $S_1, \ldots, S_m$ are algebraically independent
(see $(2.12)$ of \cite{Mac} or Corollary 7.7.2 of \cite{Sta}).
Hence, the polynomial $P_m(X)=R_m(X, \ldots, X)$ coincides with the coefficient of $T^m$
in the expansion of the right of \eqref{eq31} with $X$ in place of each $S_r$.
Thus, we see that
\begin{equation}
\sum_{n=0}^\infty P_n(X) T^n=\exp\left(-X \sum_{r=1}^\infty \frac{(-T)^r}{r}\right)
=\sum_{m=0}^\infty \binom{X}{m}T^m,
\end{equation}
which gives the Lemma.
\end{proof}

\section{Proof of the Theorems}

We begin by confirming basic properties of $\Psi_d(X)$ and $\Xi_d(X)$.
We put $\norml{\alpha +\beta\sqrt{D}}=\abs{\alpha}+\abs{\beta}\sqrt{d}$ be
the $L^1$-norm of an algebraic number $\alpha +\beta\sqrt{D}$ in $\Q(\sqrt{D})$ with $\alpha$, $\beta\in\Q$.
Let $U_d^\pm(X)$ be polynomials defined by \eqref{eq13} and expand as in \eqref{eq14}.
Clearly, we have $u_{d, 0}=1$ and,
from Lemma \ref{lm21}, $u_{d, 1}=s_{d, 1}=(\mu(d)+\sqrt{D})/2$.

Let $\Psi_d(X)=U_d^+(X) + U_d^-(X)$, $\Xi_d(X)=(U_d^-(X)-U_d^+(X))/\sqrt{D}$, and expand as in \eqref{eq12}.
Then, $\Psi_d(X)$ and $\Xi_d(X)$ are polynomials of degrees $d^\prime$ and $d^\prime -1$ respectively.
Moreover, we have
\begin{equation}
4\Phi_d(X)=4U_d^+U_d^-=(U_d^+ + U_d^-)^2-(U_d^- - U_d^+)^2=\Psi_d^2-D\Xi_d^2
\end{equation}
to confirm \eqref{eq11}.

Using notation in the previous section, we have
\begin{equation}\label{eq41}
\begin{split}
u_{d, n}= & ~ (-1)^n S^{(n)}(\zeta_d^{t_1}, \ldots, \zeta_d^{t_{d^\prime/2}})=(-1)^n R_n(s_{d, 1}, \ldots, s_{d, n}) \\
= & ~ (-1)^n \sum_{k=1}^n \sum_{\substack{e_1\leq e_2\leq \cdots \leq e_k, \\ e_1+\cdots +e_k=n}}
(-1)^{n-k} w_{(e_1, e_2, \ldots, e_k)} \prod_{i=1}^k s_{d, e_i},
\end{split}
\end{equation}
where $t_1, \ldots, t_{d^\prime/2}$ denote all integers $t$ with $1\leq t\leq d-1$ and $(t/d)=1$
and $S^{(n)}(\zeta_d^{t_1}, \ldots, \zeta_d^{t_{d^\prime/2}})$ is the value of $S^{(n)}$
evaluated at $x_i=\zeta_d^{t_i}$ for $i=1, \ldots, d^\prime/2$ and $x_i=0$ for $i>d^\prime/2$.
Since each $s_{d, k}$ belongs to $\Q(\sqrt{D})$ as is seen from Lemma \ref{lm21},
each $u_{d, n}$ also belongs to $\Q(\sqrt{D})$ by \eqref{eq41}.

Now we see that both
$a_{d, n}=u_{d, n}+\overline{u_{d, n}}$ and $b_{d, n}=(\overline{u_{d, n}}-u_{d, n})/\sqrt{D}$
are integers for $n=0, \ldots, d^\prime$.
Hence, $\Psi_d(X)$ and $\Xi_d(X)$ are polynomials of integer coefficients.
Moreover, these polynomials have the leading coefficients
$a_{d, 0}=u_{d, 0}+\overline{u_{d, 0}}=2$ and $b_{d, 1}=(\overline{u_{d, 1}}-u_{d, 1})/(2\sqrt{D})=1$ respectively.

Now we prove Theorem \ref{th11}.
As is noted in the last section, $w_{(e_1, e_2, \ldots, e_k)}$'s are all positive.
Hence, \eqref{eq41} immediately gives that
\begin{equation}
\abs{\overline{u_{d, n}}}\leq \sum_{k=0}^n v_{n, k} X_n^k,
\end{equation}
where $X_n=\max_{1\leq k\leq n}\abs{\overline{s_{d, k}}}$, and Lemma \ref{lm31} yields that
\begin{equation}
\abs{\overline{u_{d, n}}}\leq \frac{X_n(X_n+1)\cdots (X_n+n-1)}{n!}.
\end{equation}

It immediately follows from Lemma \ref{lm21} that
$\abs{\overline{s_{d, k}}}\leq \abs{1+\sqrt{D}}/2$ if $\gcd(d, k)=1$ and
$\abs{s_{d, k}}=\varphi(f)/2$ if $\gcd(d, k)=f>1$,
Hence, we have $X_n\leq F_{d, n}$ and
\begin{equation}\label{eq42}
\abs{a_{d, n}+b_{d, n}\sqrt{D}}=\abs{\overline{u_{d, n}}}
\leq \frac{F_{d, n}(F_{d, n}+1)\cdots (F_{d, n}+n-1)}{n!}.
\end{equation}
Similarly, we have $\norml{s_{d, k}}\leq G_{d, n}$ for $1\leq k\leq n$ and
\begin{equation}
\norml{a_{d, n}+b_{d, n}\sqrt{D}}=\norml{u_{d, n}}
\leq \frac{G_{d, n}(G_{d, n}+1)\cdots (G_{d, n}+n-1)}{n!},
\end{equation}
proving Theorem \ref{th11}.

Now we prove the Corollary.
Jameson's explicit version of Stirling's formula in \cite{Jam} gives that
\begin{equation}
\sqrt{2\pi}x^{x+1/2}e^{-x}<\Gamma(x+1)<\sqrt{2\pi}x^{x+1/2}e^{-x+1/(12x)}
\end{equation}
for $x>1$.
Hence, we obtain
\begin{equation}\label{eq43}
\begin{split}
\log \frac{\Gamma(F+n)}{\Gamma(n+1)\Gamma(F)}
= ~ &\left(F-\frac{1}{2}\right)\log\frac{F+n-1}{F-1}+n\log\left(1+\frac{F-1}{n}\right) \\
~ & -\frac{\log(2\pi n)}{2}+\frac{1}{12(F+n)} \\
< ~ & \left(F-\frac{1}{2}\right)\left(1+\log\frac{F+n-1}{F-1}\right) \\
~ & -\frac{1+\log(2\pi n)}{2}+\frac{1}{12(F+n)}.
\end{split}
\end{equation}
for $F>1$.
From \eqref{eq42}, we have
\begin{equation}
\abs{u_{d, n}}\leq \frac{\Gamma((F_{d, n}+n)}{\Gamma(F_{d, n})\Gamma(n+1)}
<\frac{e^{1/(12(F_{d, n}+n))}}{\sqrt{2e\pi n}}
\left(\frac{e(F_{d, n}+n-1)}{F_{d, n}-1}\right)^{F_{d, n}-\frac{1}{2}},
\end{equation}
which gives the first bound in \eqref{eq100}.
Interchanging roles of $F$ and $n+1$ in \eqref{eq43}, we obtain the second bound in \eqref{eq100}.
Observing that
\begin{equation}
\frac{\alpha (\alpha -1)\cdots (\alpha -n+1)}{n!}x^k<\sum_{k=0}^\infty \binom{\alpha}{k} x^k=(1+x)^\alpha
\end{equation}
for $\alpha, x>0$, we obtain $\abs{u_{d, n}}<2^{F_{d, n}+n-1}$ and then the last bound in \eqref{eq100}.
Now the Corollary follows.

Finally, we prove Theorem \ref{th12}.
Theorem \ref{th11} gives that
\begin{equation}\label{eq51}
\norml{u_{d, n}}\leq 2\binom{G_d+n-1}{n}
\end{equation}
for $0\leq n\leq \varphi(d)/4$.
Since $u_{d, d^\prime -n}=\pm u_{d, n}$ or $\pm \overline{u_{d, n}}$ for $n=0, \ldots, d^\prime$
and $u_{d, 0}=2$, \eqref{eq51} holds for $n=0, \ldots, d^\prime$.
Hence, we see that
\begin{equation}\label{eq52}
\begin{split}
\Psi_d(x)\geq & ~ 2x^{d^\prime}-\sum_{n=1}^{d^\prime}\abs{a_{d, n}}x^{d^\prime -n}
>2x^{d^\prime}\left(1-\sum_{n=1}^{d^\prime}\binom{G_d+n-1}{n} x^{-n}\right) \\
\geq & ~ 2x^{d^\prime}\left(1-\sum_{n=1}^{d^\prime}\left(\frac{G_d}{x}\right)^n\right).
\end{split}
\end{equation}
Since we have assumed that $x>2G_d$, we have $\Psi_d(x)>0$.

Now we see that
\begin{equation}\label{eq53}
\abs{\frac{\Xi_d(x)\sqrt{D}}{\Psi_d(x)}}
\leq \frac{\left(x^{d^\prime -1}+\sum_{n=2}^{d^\prime } \abs{b_{d, n}}x^{d^\prime -n}\right)\sqrt{d}}
{2x^{d^\prime} -\mu(d)x^{d^\prime -1}-\sum_{n=2}^{d^\prime} \abs{a_{d, n}} x^{d^\prime -n}},
\end{equation}
where we note that the denominator of the right-hand side is positive as in \eqref{eq52}.
With the aid of \eqref{eq51}, we have
\begin{equation}
(1+\sqrt{d})x^{d^\prime -1}+\sum_{n=2}^{d^\prime} \norml{u_{d, n}} x^{d^\prime -n}<2x^{d^\prime},
\end{equation}
and therefore
\begin{equation}
\left(x^{d^\prime -1}+\sum_{n=2}^{d^\prime} \abs{b_{d, n}}x^{d^\prime -n}\right)\sqrt{d}
<2x^{d^\prime} -\mu(d)x^{d^\prime -1}-\sum_{n=2}^{d^\prime} \abs{a_{d, n}} x^{d^\prime -n}.
\end{equation}
Hence, \eqref{eq53} yields that
\begin{equation}
\begin{split}
\abs{\frac{\Xi_d(x)\sqrt{D}}{\Psi_d(x)}}
\leq ~ & \frac{x^{d^\prime -1}\sqrt{d}+\sum_{n=2}^{d^\prime} \norml{u_{d, n}} x^{d^\prime -n}}
{2x^{d^\prime}-\mu(d)x^{d^\prime-1}} \\
= ~ & \frac{1}{2x-\mu(d)}\left(\sqrt{d}+\sum_{n=2}^{d^\prime} \frac{\norml{u_{d, n}}}{x^{n-1}}\right).
\end{split}
\end{equation}

Proceeding as above, we use \eqref{eq51} to obtain
\begin{equation}
\begin{split}
\abs{\frac{\Xi_d(x)\sqrt{D}}{\Psi_d(x)}}
< ~ & 
\frac{1}{2x-\mu(d)}\left(\sqrt{d}+\sum_{n=2}^{d^\prime} \binom{G_d+n-1}{n}\frac{1}{x^{n-1}}\right) \\
< ~ & \frac{1}{2x-\mu(d)}\left(\sqrt{d}+x\left(\left(1-\frac{1}{x}\right)^{-G_d}-1-\frac{G_d}{x}\right)\right),
\end{split}
\end{equation}
which is the desired upper bound.
The lower bound can be derived from the observation that
\begin{equation}
\abs{\frac{\Xi_d(x)\sqrt{D}}{\Psi_d(x)}}
\geq \frac{\left(x^{d^\prime -1}-\sum_{n=2}^{d^\prime} \abs{b_{d, n}}x^{d^\prime -n}\right)\sqrt{d}}
{2x^{d^\prime} -\mu(d)x^{d^\prime -1}+\sum_{n=2}^{d^\prime} \abs{a_{d, n}} x^{d^\prime -n}}
\end{equation}
in a similar way as above.
Thus Theorem \ref{th12} is proved.

{}
\end{document}